\documentclass{my_aims}

\usepackage{amsmath}
\usepackage{paralist}

\usepackage[colorlinks=true]{hyperref}
\hypersetup{urlcolor=blue, citecolor=red}

\newtheorem{theorem}{Theorem}[section]

\theoremstyle{definition}
\newtheorem{definition}[theorem]{Definition}

\title[Noether's theorem for H-O problems of Herglotz]{%
Noether's theorem for higher-order variational problems of Herglotz type}

\author[Sim\~{a}o P. S. Santos, Nat\'{a}lia Martins and Delfim F. M. Torres]{}



\email{spsantos@ua.pt}
\email{natalia@ua.pt}
\email{delfim@ua.pt}

\thanks{This work is part of first author's Ph.D.,
which is carried out at the University of Aveiro.}

\begin{document}

\maketitle

\centerline{\scshape Sim\~{a}o P. S. Santos, Nat\'{a}lia Martins and Delfim F. M. Torres}
\medskip
{\footnotesize
\centerline{Center for Research and Development in Mathematics and Applications (CIDMA)}
\centerline{Department of Mathematics, University of Aveiro, 3810-193 Aveiro, Portugal}}

\bigskip


\begin{abstract}
We approach higher-order variational problems of Herglotz type from an optimal
control point of view. Using optimal control theory, we derive a generalized
Euler--Lagrange equation, transversality conditions, DuBois--Reymond necessary
optimality condition and Noether's theorem for Herglotz's type higher-order
variational problems, valid for piecewise smooth functions.
\end{abstract}


\section{Introduction}

The generalized variational problem proposed
by Herglotz in 1930 \cite{Herglotz1930} can be formulated as follows:
\begin{equation}
\label{PH_1}
\tag{$H_1$}
\begin{gathered}
z(b)\longrightarrow \text{extr} \\
\text{with } \dot{z}(t)=L(t,x(t),\dot{x}(t),z(t)), \quad t \in [a,b],\\
\text{subject to }  z(a)=\gamma, \quad  \gamma \in \mathbb{R}.
\end{gathered}
\end{equation}
The Herglotz variational problem consists in the determination of trajectories
$x(\cdot)$ subject to some initial condition $x(a)=\alpha$, $\alpha \in \mathbb{R}$,
that extremize (maximize or minimize) the value $z(b)$,
where $L \in C^1([a,b]\times \mathbb{R}^{2m+1};\mathbb{R})$.
While in \cite{Herglotz1930} the admissible functions
are $x(\cdot) \in C^2([a,b];\mathbb{R}^m)$ and $z(\cdot) \in C^1([a,b];\mathbb{R})$,
here we consider \eqref{PH_1} in the wider class of functions
$x(\cdot) \in PC^1([a,b];\mathbb{R}^m)$ and $z(\cdot) \in PC^1([a,b];\mathbb{R})$,
where the notation $PC$ stands for ``piecewise continuous'' (for the precise
meaning of piecewise continuity and piecewise differentiability see, e.g.,
\cite[Sec.~1.1]{MR2316829}).

It is clear to see that Herglotz's problem \eqref{PH_1} reduces to the classical
fundamental problem of the calculus of variations
(see, e.g., \cite{vanBrunt}) if the Lagrangian $L$ does not depend on the
$z$ variable: if $\dot{z}(t)=L(t,x(t),\dot{x}(t))$, $t \in [a,b]$,
then \eqref{PH_1} is equivalent to the classical variational problem
\begin{equation}
\label{eq:class:Funct}
\int_a^b L(t,x(t),\dot{x}(t))dt \longrightarrow \textrm{extr}.
\end{equation}
Herglotz proved that a necessary optimality
condition for a pair $\left(x(\cdot),z(\cdot)\right)$ to be an extremizer
of the generalized variational problem \eqref{PH_1} is given by
\begin{multline}
\label{eq:gen:EL}
\frac{\partial L}{\partial x}\left(t,x(t),\dot{x}(t),z(t)\right)
-\frac{d}{dt}\frac{\partial L}{\partial \dot{x}}\left(t,x(t),\dot{x}(t),z(t)\right)\\
+\frac{\partial L}{\partial z}\left(t,x(t),\dot{x}(t),z(t)\right)
\frac{\partial L}{\partial \dot{x}}\left(t,x(t),\dot{x}(t),z(t)\right) = 0,
\end{multline}
$t \in [a,b]$. The equation \eqref{eq:gen:EL} is known as the generalized
Euler--Lagrange equation. Observe that for the classical problem
of the calculus of variations \eqref{eq:class:Funct} one has
$\frac{\partial L}{\partial z}=0$ and equation
\eqref{eq:gen:EL} reduces to the classical Euler--Lagrange equation
$$
\frac{\partial L}{\partial x}\left(t,x(t),\dot{x}(t)\right)
-\frac{d}{dt}\frac{\partial L}{\partial \dot{x}}\left(t,x(t),\dot{x}(t)\right) =0.
$$

In \cite{Simao+NM+Torres2013} we have introduced higher-order variational problems
of Herglotz type and obtained a generalized Euler--Lagrange equation and transversality
conditions for these problems. In particular, we considered the problem
of determining the trajectories $x(\cdot)$ such that
\begin{equation}
\label{PH_n}
\tag{$H_n$}
\begin{gathered}
z(b)\longrightarrow \text{extr}\\
\text{with }\dot{z}(t)=L(t,x(t),\dot{x}(t),\ldots,x^{(n)}(t),z(t)), \quad t \in [a,b],\\
\text{subject to } z(a)=\gamma,\quad \gamma \in \mathbb{R}.
\end{gathered}
\end{equation}
We proved that if a pair $(x(\cdot),z(\cdot))$ is an extremizer of the higher-order problem
\eqref{PH_n}, then it satisfies the higher-order generalized Euler--Lagrange equation
\begin{equation*}
\sum_{j=0}^n(-1)^j\frac{d^j}{dt^j}\left(\psi_z(t)
\frac{\partial L}{\partial x^{(j)}}\langle x, z \rangle_n(t)\right)=0,
\quad t \in [a,b],
\end{equation*}
and the transversality conditions $\psi_j(b)=\psi_j(a)=0$, for $j=1, \ldots, n$, where
\begin{equation*}
\begin{cases}
\psi_z(t)=e^{\int_t^b\frac{\partial L}{\partial z}\langle x, z \rangle_n(\theta)d\theta}\\
\psi_j(t)=\sum_{i=0}^{n-j}(-1)^{i+1}\frac{d^i}{dt^i}\left(\psi_z(t)
\frac{\partial L}{\partial x^{(i+j)}}\langle x, z \rangle_n(t)\right),
\quad j=1,\ldots,n.
\end{cases}
\end{equation*}
While in \cite{Simao+NM+Torres2013} the admissible functions are
$x(\cdot)\in C^{2n}([a,b]; \mathbb{R}^m)$ and $z(\cdot) \in C^1([a,b];\mathbb{R})$,
here we consider \eqref{PH_n} in the wider class of functions
$x(\cdot)\in PC^{n}([a,b]; \mathbb{R}^m)$ and $z(\cdot) \in PC^1([a,b];\mathbb{R})$.

One of the most important results in optimal control theory is Pontryagin's
maximum principle proved by Pontryagin et al. in \cite{Pontryagin}.
This principle provides conditions for optimization problems with
differential equations as constraints. The maximum principle is still
widely used for solving problems of control and other problems of dynamic optimization.
Moreover, basic necessary optimality conditions from classical
calculus of variations follow from Pontryagin's maximum principle.

One of the problems of optimal control, in Bolza form, is the following one:
\begin{equation}
\label{problem P}
\tag{$P$}
\begin{gathered}
\mathcal{J}(x(\cdot),u(\cdot))=\int_a^b f(t,x(t),u(t))dt
+\phi(x(b))\longrightarrow\text{extr}\\
\text{subject to } \dot{x}(t)=g(t,x(t),u(t)),
\end{gathered}
\end{equation}
with some initial condition on $x$,
where $f \in C^1([a,b]\times \mathbb{R}^{m}\times \Omega;\mathbb{R})$,
$\phi \in C^1(\mathbb{R}^{m};\mathbb{R})$,
$g \in C^1([a,b]\times \mathbb{R}^{m}\times \Omega;\mathbb{R}^m)$,
$x \in PC^1([a,b]; \mathbb{R}^m)$ and $u\in PC([a,b];\Omega)$,
with $\Omega \subseteq \mathbb{R}^r$ an open set. In the literature
of optimal control, $x$ and $u$ are frequently called the state
and control variables, respectively, while $\phi$ is known
as the payoff or salvage term. Note that the classical
problem of the calculus of variations \eqref{eq:class:Funct}
is a particular case of problem \eqref{problem P}
with $\phi(x) \equiv 0$, $g(t,x,u)=u$ and $\Omega=\mathbb{R}^m$.
In this work we show how the results on the higher-order
variational problem of Herglotz \eqref{PH_n} obtained in
\cite{Simao+NM+Torres2013} can be generalized by using the theory of optimal control.
The technique used consists in rewriting the generalized higher-order variational
problem of Herglotz \eqref{PH_n} as a standard optimal control problem \eqref{problem P},
and then to apply available results of optimal control theory.
For the first-order case we refer the reader to \cite{MyID:313}.

The paper is organized as follows. In Section~\ref{sec:prelim}
we present the necessary concepts and results from optimal control theory:
Pontryagin's maximum principle (Theorem~\ref{PMP});
the DuBois--Reymond condition of optimal control (Theorem~\ref{thm DB-R OC});
and the Noether theorem of optimal control (Theorem~\ref{thm opt cont Noether}).
Our main results are given in Section~\ref{sec:MainRes}:
we extend the higher-order Euler--Lagrange equation and the transversality conditions
for problem \eqref{PH_n} found in \cite{Simao+NM+Torres2013} to admissible
functions $x(\cdot) \in PC^n([a,b];\mathbb{R}^m)$ and $z(\cdot) \in PC^1([a,b];\mathbb{R})$
(Theorem~\ref{thm H-O EL+TC}); we obtain a DuBois--Reymond necessary optimality condition
for problem \eqref{PH_n} (Theorem~\ref{thm DB-R PH_n}); and we generalize the Noether theorem
to higher-order variational problems of Herglotz type (Theorem~\ref{thm noether PH_n}).
We end with Section~\ref{sec:conc} of conclusions.


\section{Preliminaries}
\label{sec:prelim}

We begin this section by stating the well known Pontryagin's maximum principle,
which is a first-order necessary optimality condition.

\begin{theorem}[Pontryagin's maximum principle for problem \eqref{problem P} \cite{Pontryagin}]
\label{PMP}
If a pair $(x(\cdot),u(\cdot))$ with $x \in PC^1([a,b]; \mathbb{R}^m)$ and $u\in PC([a,b];\Omega)$
is a solution to problem \eqref{problem P} with the initial condition $x(a)=\alpha$,
$\alpha \in \mathbb{R}$, then there exists $\psi \in PC^1([a,b];\mathbb{R}^m)$
such that the following conditions hold:
\begin{itemize}

\item the optimality condition
\begin{equation}
\label{prob P opt condt}
\frac{\partial H}{\partial u}(t, x(t),u(t), \psi(t))=0;
\end{equation}

\item the adjoint system
\begin{equation}
\label{prob P adj syst}
\begin{cases}
\dot{x}(t)=\frac{\partial H}{\partial \psi}(t, x(t),u(t), \psi(t))\\
\dot{\psi}(t)=-\frac{\partial H}{\partial x}(t, x(t),u(t), \psi(t));
\end{cases}
\end{equation}

\item and the transversality condition
\begin{equation}
\label{prob P tr condt}
\psi(b)=grad(\phi(x))(b);
\end{equation}
\end{itemize}
where the Hamiltonian $H$ is defined by
\begin{equation}
\label{eq:def:Hamiltonian}
H(t,x,u,\psi)=f(t,x,u)+\psi\cdot g(t,x,u).
\end{equation}
\end{theorem}

\begin{definition}[Pontryagin extremal to \eqref{problem P}]
A triplet $(x(\cdot),u(\cdot), \psi(\cdot))$ with
$x \in PC^1([a,b]; \mathbb{R}^m)$, $u\in PC([a,b];\Omega)$ and
$\psi \in PC^1([a,b];\mathbb{R}^m)$ is called a Pontryagin extremal
to problem \eqref{problem P} if it satisfies the optimality condition
\eqref{prob P opt condt}, the adjoint system \eqref{prob P adj syst}
and the transversality condition \eqref{prob P tr condt}.
\end{definition}

\begin{theorem}[DuBois--Reymond condition of optimal control \cite{Pontryagin}]
\label{thm DB-R OC}
If $(x(\cdot),u(\cdot), \psi(\cdot))$ is a Pontryagin extremal
to problem \eqref{problem P}, then the Hamiltonian \eqref{eq:def:Hamiltonian}
satisfies the equality
\begin{equation*}
\frac{d H}{dt}(t,x(t),u(t),\psi(t))
=\frac{\partial H}{\partial t}(t,x(t),u(t),\psi(t)),
\quad 	t \in [a,b].
\end{equation*}
\end{theorem}

The famous Noether theorem \cite{Noether1918}
is another fundamental tool of the calculus of variations \cite{MR2098297},
optimal control \cite{Torres2001,Torres2002,Torres2004}
and modern theoretical physics \cite{MR3072684}.
It states that when an optimal control problem is invariant
under a one parameter family of transformations,
then there exists a corresponding conservation law:
an expression that is conserved along all the Pontryagin
extremals of the problem (see \cite{Torres2001,Torres2002,Torres2004} and references therein).
Here we use Noether's theorem as found in \cite{Torres2001}, which is formulated
for problems of optimal control in Lagrange form, that is, for problem
\eqref{problem P} with $\phi \equiv 0$. In order to apply the results of \cite{Torres2001}
to the Bolza problem \eqref{problem P}, we rewrite it in the following equivalent Lagrange form:
\begin{equation}
\label{eq:prob:Lag}
\begin{gathered}
\mathcal{I}(x(\cdot),y(\cdot),u(\cdot))
=\int_a^b \left[f(t,x(t),u(t))+ y(t)\right] dt \longrightarrow\text{extr},\\
\begin{cases}
\dot{x}(t)=g\left(t,x(t),u(t)\right),\\
\dot{y}(t)=0,
\end{cases}\\
x(a)=\alpha, \ y(a)= \frac{\phi(x(b))}{b-a}.
\end{gathered}
\end{equation}
Before presenting the Noether theorem for the optimal control problem \eqref{problem P},
we need to define the concept of invariance. Here we apply the notion of invariance
found in \cite{Torres2001} to the equivalent optimal control problem \eqref{eq:prob:Lag}.
In Definition~\ref{def inv OC} we use the little-o notation.

\begin{definition}[Invariance of problem \eqref{problem P} cf. \cite{Torres2001}]
\label{def inv OC}
Let $h^s$ be a one-parameter family of invertible $C^1$ maps
\begin{equation*}
\begin{gathered}
h^s:[a,b]\times \mathbb{R}^m\times \Omega \longrightarrow\mathbb{R}\times\mathbb{R}^m\times \mathbb{R}^r,\\
h^s(t,x,u)=\left(\mathcal{T}^s(t,x,u), \mathcal{X}^s(t,x,u),\mathcal{U}^s(t,x,u)\right),\\
h^0(t,x,u)=(t,x,u) \text{ for all } (t,x,u)\in [a,b]\times \mathbb{R}^m\times \Omega.
\end{gathered}
\end{equation*}
Problem \eqref{problem P} is said to be invariant under transformations $h^s$
if for all $(x(\cdot),u(\cdot))$ the following two conditions hold:
\begin{enumerate}
\item[(i)]
\begin{multline}
\label{eq inv OC 1}
\left[f \circ h^s(t,x(t),u(t))+\frac{\phi(x(b))}{b-a} + \xi s
+ o(s)\right]\frac{d\mathcal{T}^s}{dt}(t,x(t),u(t))\\
= f(t,x(t),u(t)) + \frac{\phi(x(b))}{b-a}
\end{multline}
for some constant $\xi$;

\item[(ii)]
\begin{equation}
\label{eq inv OC 2}
\frac{d\mathcal{X}^s}{dt}\left(t,x(t),u(t)\right)
=g\circ h^s(t,x(t),u(t))\frac{d\mathcal{T}^s}{dt}(t,x(t),u(t)).
\end{equation}
\end{enumerate}
\end{definition}

The next result can be easily obtained from the Noether theorem proved
by Torres in \cite{Torres2001} and Pontryagin's maximum principle (Theorem \ref{PMP}).

\begin{theorem}[Noether's theorem for the optimal control problem \eqref{problem P}]
\label{thm opt cont Noether}
If problem \eqref{problem P} is invariant
in the sense of Definition~\ref{def inv OC}, then the quantity
\begin{equation*}
(b-t) \xi + \psi(t) \cdot X(t,x(t),u(t))
-\left[H(t,x(t),u(t),\psi(t)) + \frac{\phi(x(b))}{b-a}\right]\cdot T(t,x(t),u(t))
\end{equation*}
is constant in $t$ along every Pontryagin extremal $(x(\cdot),u(\cdot), \psi(\cdot))$
of problem \eqref{problem P}, where $H$ is defined by \eqref{eq:def:Hamiltonian} and
\begin{equation*}
\begin{gathered}
T(t,x(t),u(t))=\frac{\partial \mathcal{T}^s}{\partial s}(t,x(t),u(t))\biggm\vert_{s=0},\\
X(t,x(t),u(t))=\frac{\partial \mathcal{X}^s}{\partial s}(t,x(t),u(t))\biggm\vert_{s=0}.
\end{gathered}
\end{equation*}
\end{theorem}


\section{Main results}
\label{sec:MainRes}

We begin by introducing some definitions for the
higher-order variational problem of Herglotz \eqref{PH_n}.

\begin{definition}[Admissible pair to problem \eqref{PH_n}]
We say that $(x(\cdot),z(\cdot))$ with $x(\cdot) \in PC^n([a,b];\mathbb{R}^m)$
and $z(\cdot) \in PC^1([a,b];\mathbb{R})$ is an admissible pair to problem
\eqref{PH_n} if it satisfies the equation
\begin{equation*}
\begin{gathered}
\dot{z}(t)=L(t,x(t),\dot{x}(t),\cdots, x^{(n)}(t),z(t)), \quad t \in [a,b],\\
\text{with } z(a)=\gamma \in \mathbb{R}.
\end{gathered}
\end{equation*}
\end{definition}

\begin{definition}[Extremizer to problem \eqref{PH_n}]
We say that an admissible pair $(x^*(\cdot),z^*(\cdot))$ is an extremizer
to problem \eqref{PH_n} if $z(b)-z^*(b)$ has the same signal for all admissible
pairs $(x(\cdot),z(\cdot))$ that satisfy $\|z-z^* \|_0< \epsilon$ and $\|x-x^* \|_0< \epsilon$
for some positive real $\epsilon$, where $\|y\|_0=\smash{\displaystyle\max_{a\leq t \leq b}}|y(t)|$.
\end{definition}

We now present a necessary condition for a pair $(x(\cdot),z(\cdot))$
to be a solution (extremizer) to problem \eqref{PH_n}. The following result generalizes
\cite{Simao+NM+Torres2013} by considering a more general class of functions.
To simplify notation, we use the operator $\langle\cdot,\cdot\rangle_n$,
$n \in \mathbb{N}$, defined by
$\langle x, z \rangle_n(t):=(t,x(t),\dot{x}(t), \ldots, x^{(n)}(t),z (t))$.
When there is no possibility of ambiguity, we sometimes suppress arguments.

\begin{theorem}[Higher-order Euler--Lagrange equation and transversality conditions for problem \eqref{PH_n}]
\label{thm H-O EL+TC}
If $(x(\cdot),z(\cdot))$ is an extremizer to problem \eqref{PH_n} that satisfies the inicial conditions
\begin{equation}
\label{init cond HO}
x(a)=\alpha_0, \ldots, x^{(n-1)}(a)=\alpha_{n-1}, \quad \alpha_0,\ldots, \alpha_{n-1} \in \mathbb{R},
\end{equation}
then the Euler--Lagrange equation
\begin{equation}
\label{H-O EL}
\sum_{j=0}^n(-1)^j\frac{d^j}{dt^j}\left(\psi_z(t)
\frac{\partial L}{\partial x^{(j)}}\langle x, z \rangle_n(t)\right)=0
\end{equation}
holds, for $t \in [a,b]$, where
\begin{equation}
\label{multipliers}
\begin{cases}
\psi_z(t)=e^{\int_t^b\frac{\partial L}{\partial z}\langle x, z \rangle_n(\theta)d\theta}\\
\psi_j(t)=\sum_{i=0}^{n-j}(-1)^{i+1}\frac{d^i}{dt^i}\left(\psi_z(t)
\frac{\partial L}{\partial x^{(i+j)}}\langle x, z \rangle_n(t)\right),
\quad j=1,\ldots,n.
\end{cases}
\end{equation}
Moreover, the following transversality conditions hold:
\begin{equation}
\label{H-O tr condt}
\psi_j(b)=0, \quad j=1, \ldots, n.
\end{equation}
\end{theorem}

\begin{proof}
Observe that the higher-order problem of Herglotz \eqref{PH_n} is a particular case
of problem \eqref{problem P} when we consider a $n+1$ coordinates state variable
$(x_0,x_1,\ldots x_{n-1},z)$ with $x_0=x$, $x_1=\dot{x}$, \ldots, $x_{n-1}=x^{(n-1)}$,
a control $u=x^{(n)}$ and choose $f\equiv 0$ and $\phi(x_0, \ldots, x_{n-1},z)=z$.
The higher-order problem of Herglotz can now be described
as an optimal control problem as follows:
\begin{equation}
\label{PH_n_OC}
\begin{gathered}
z(b)\longrightarrow \text{extr}\\
\begin{cases}
\dot{x}_0(t)=x_1(t),\\
\dot{x}_1(t)=x_2(t),\\
\dot{x}_2(t)=x_3(t),\\
\quad\quad\quad \vdots\\
\dot{x}_{n-2}(t)=x_{n-1}(t),\\
\dot{x}_{n-1}(t)=u(t),\\
\dot{z}(t)=L(t,x_0(t), \ldots, x_{n-1}(t),u(t),z(t)),
\end{cases}\\
z(a)=\gamma,  \quad \gamma \in \mathbb{R}.
\end{gathered}
\end{equation}
Observe that since we consider problem \eqref{PH_n} subject
to the initial conditions \eqref{init cond HO}, then
$x_0(a)=\alpha_0,\ldots, x_{n-1}(a)=\alpha_{n-1}$	
with $\alpha_0, \ldots, \alpha_{n-1}$ given real numbers.
From Pontryagin's Maximum Principle for problem \eqref{problem P}
(Theorem~\ref{PMP}) there are $(\psi_1, \ldots, \psi_{n},\psi_z)
\in PC^1([a,b];\mathbb{R}^{n\times m+1})$ such that the following conditions hold:
\begin{itemize}

\item the optimality condition
\begin{equation}
\label{opt_cond_HO}
\frac{\partial H}{\partial u}(t,x_0(t),\ldots, x_{n-1}(t),u(t),z(t),
\psi_1(t),\ldots,\psi_{n}(t),\psi_z(t))=0,
\end{equation}

\item the adjoint system
\begin{equation}
\label{adj_syst_HO}
\begin{cases}
\dot{x}_{j-1}(t)=\frac{\partial H}{\partial \psi_{j}}(t,x_0(t),
\ldots, x_{n-1}(t),u(t),z(t),\psi_1(t),\ldots,\psi_{n}(t),\psi_z(t)),
\  j=1,\ldots,n\\
\dot{\psi}_1(t)=-\frac{\partial H}{\partial x_0}(t,x_0(t),
\ldots, x_{n-1}(t),u(t),z(t),\psi_1(t),\ldots,\psi_{n}(t),\psi_z(t))\\
\dot{\psi}_j(t)=-\frac{\partial H}{\partial x_{j-1}}(t,x_0(t),
\ldots, x_{n-1}(t),u(t),z(t),\psi_1(t),\ldots,\psi_{n}(t),\psi_z(t)),
\  j=2,\ldots,n\\
\dot{\psi}_z(t)=-\frac{\partial H}{\partial z}(t,x_0(t),
\ldots, x_{n-1}(t),u(t),z(t),\psi_1(t),\ldots,\psi_{n}(t),\psi_z(t))
\end{cases}
\end{equation}

\item the transversality conditions
\begin{equation}
\label{tr_cond_HO}
\begin{cases}
\psi_i(b)=0, \quad j=1, \ldots, n,\\
\psi_z(b)=1,
\end{cases}
\end{equation}
\end{itemize}
where the Hamiltonian $H$ is defined by
\begin{multline*}
H(t,x_0, \ldots, x_{n-1},u,z,\psi_1,\ldots,\psi_{n},\psi_z)\\
=\psi_1\cdot x_1+\ldots+\psi_{n-1}\cdot x_{n-1}+\psi_{n}\cdot u
+\psi_z\cdot L(t,x_0, \ldots, x_{n-1},u,z).
\end{multline*}
Observe that the optimality condition \eqref{opt_cond_HO} implies that
$\psi_{n}=-\psi_z\frac{\partial L}{\partial u}$ and that the adjoint system
\eqref{adj_syst_HO} implies that
\begin{equation*}
\begin{cases}
\dot{\psi}_1=-\psi_z\frac{\partial L}{\partial x_0},\\
\dot{\psi}_j=-\psi_{j-1}-\psi_z\frac{\partial L}{\partial x_{j-1}},
\quad \text{ for } j=2,\ldots,n,\\
\dot{\psi}_z=-\psi_z\frac{\partial L}{\partial z}.
\end{cases}
\end{equation*}
Hence, $\psi_z$ is solution of a first-order linear differential equation,
which is solved using an integrand factor to find that
$\psi_z(t)=ke^{-\int_a^t\frac{\partial L}{\partial z}d\theta}$ with $k$ a constant.
From the last transversality condition in \eqref{tr_cond_HO}, we obtain that
$k=e^{\int_a^b\frac{\partial L}{\partial z}d\theta}$ and, consequently,
$$
\psi_z(t)=e^{\int_t^b\frac{\partial L}{\partial z}d\theta}.
$$
Note also that for $j=n$ we obtain
$\dot{\psi}_{n}=-\psi_{n-1}-\psi_z\frac{\partial L}{\partial x_{n-1}}$,
which is equivalent to
\begin{equation*}
\psi_{n-1}=\frac{d}{dt}\left(\psi_z \frac{\partial L}{\partial x_n}\right)
-\psi_z\frac{\partial L}{\partial x_{n-1}}.
\end{equation*}
By differentiation of the previous expression, we obtain that
\begin{equation*}
\dot{\psi}_{n-1}=\frac{d^2}{dt^2}\left(\psi_z \frac{\partial L}{\partial x_n}\right)
-\frac{d}{dt}\left(\psi_z\frac{\partial L}{\partial x_{n-1}}\right)
\end{equation*}
and noting that $\dot{\psi}_{n-1}=-\psi_{n-2}-\psi_z\frac{\partial L}{\partial x_{n-2}}$,
we find an expression for $\psi_{n-2}$:
$$
\psi_{n-2}=-\frac{d^2}{dt^2}\left(\psi_z \frac{\partial L}{\partial x_n}\right)
+\frac{d}{dt}\left(\psi_z\frac{\partial L}{\partial x_{n-1}}\right)
-\psi_z\frac{\partial L}{\partial x_{n-2}}.
$$
Similarly, we obtain that
$$
\psi_{n-3}=\frac{d^3}{dt^3}\left(\psi_z \frac{\partial L}{\partial x_n}\right)
-\frac{d^2}{dt^2}\left(\psi_z\frac{\partial L}{\partial x_{n-1}}\right)
+\frac{d}{dt}\left(\psi_z\frac{\partial L}{\partial x_{n-2}}\right)
-\psi_z\frac{\partial L}{\partial x_{n-3}}.
$$
Applying the same argument to the next multipliers and noting that
$\psi_1=-\dot{\psi}_2-\psi_z\frac{\partial L}{\partial x_1}$, we have
\begin{multline*}
\dot{\psi}_1=-\psi_z\frac{\partial L}{\partial x_0}\\
=(-1)^{n}\frac{d^{n}}{dt^{n}}\left(\psi_z\frac{\partial L}{\partial x_{n}}\right)
+(-1)^{n-1}\frac{d^{n-1}}{dt^{n-1}}\left(\psi_z\frac{\partial L}{\partial x_{n-1}}\right)
+\cdots-\frac{d}{dt}\left(\psi_z\frac{\partial L}{\partial x_1}\right)
\end{multline*}
or, equivalently,
\begin{equation*}
(-1)^{n}\frac{d^{n}}{dt^{n}}\left(\psi_z\frac{\partial L}{\partial x_{n}}\right)
+(-1)^{n-1}\frac{d^{n-1}}{dt^{n-1}}\left(\psi_z\frac{\partial L}{\partial x_{n-1}}\right)
+\cdots-\frac{d}{dt}\left(\psi_z\frac{\partial L}{\partial x_1}\right)
+\psi_z\frac{\partial L}{\partial x_0}=0.
\end{equation*}
Rewriting previous equation in terms of problem \eqref{PH_n}
and in the form of a summation, one gets
$$
\sum_{j=0}^n(-1)^j\frac{d^j}{dt^j}\left(\psi_z\frac{\partial L}{\partial x^{(j)}}\right)=0
$$
as intended. Observe also that from the previous argumentation
we were also able to derive expressions for the multipliers:
\begin{equation*}
\psi_j=\sum_{i=0}^{n-j}(-1)^i\frac{d^i}{dt^i}\left(
-\psi_z\frac{\partial L}{\partial x^{(i+j)}}\right),
\quad j=1,\ldots,n,
\end{equation*}
which together with \eqref{tr_cond_HO} leads to the transversality conditions
\begin{equation*}
\sum_{i=0}^{n-j}(-1)^i\frac{d^i}{dt^i}\left(\psi_z
\frac{\partial L}{\partial x^{(i+j)}}\right)\bigg\vert_{t=b}=0,
\quad j=1, \ldots, n.
\end{equation*}
This concludes the proof.
\end{proof}

\begin{definition}[Extremal to problem \eqref{PH_n}]
We say that an admissible pair $(x(\cdot), z(\cdot))$ is an extremal
to problem \eqref{PH_n} if it satisfies the Euler--Lagrange equation
\eqref{H-O EL} and the transversality conditions \eqref{H-O tr condt}.
\end{definition}

Next we present two new important results: the DuBois--Reymond condition
and the Noether theorem for the higher-order variational problem of Herglotz \eqref{PH_n}.

\begin{theorem}[DuBois--Reymond condition for problem \eqref{PH_n}]
\label{thm DB-R PH_n}
If $(x(\cdot), z(\cdot))$ is an extremal to problem \eqref{PH_n}, then
\begin{equation*}
\frac{d}{dt}\left(\sum_{i=1}^n\psi_i(t) x^{(i)}(t)
+\psi_z(t) L\langle x, z \rangle_n(t) \right)
=\psi_z(t)\frac{\partial L}{\partial t}\langle x, z \rangle_n(t),
\end{equation*}
where $\psi_z(t)$ and $\psi_i(t)$  are defined in \eqref{multipliers}.
\end{theorem}

\begin{proof}
Rewrite \eqref{PH_n} as optimal control problem \eqref{PH_n_OC}
and apply Theorem~\ref{thm DB-R OC}.
\end{proof}

\begin{definition}[Invariance for problem \eqref{PH_n}]
\label{def inv_PH_n}
Let $h^s$ be a one-parameter family of invertible $C^1$ maps
$h^s:[a,b]\times \mathbb{R}^m \times \mathbb{R}
\longrightarrow \mathbb{R}\times \mathbb{R}^m \times \mathbb{R}$,
\begin{equation*}
\begin{gathered}
h^s(t,x(t),z(t))=(\mathcal{T}^s\langle x,z \rangle_n(t),
\mathcal{X}^s\langle x,z \rangle_n(t),\mathcal{Z}^s\langle x,z \rangle_n(t)),\\
h^0(t,x,z)=(t,x,z), \quad \forall (t,x,z) \in [a,b]\times \mathbb{R}^m  \times \mathbb{R}.
\end{gathered}
\end{equation*}
Problem \eqref{PH_n} is said to be invariant under the transformations $h^s$
if for all admissible pairs $(x(\cdot),z(\cdot))$ the following two conditions hold:
\begin{enumerate}
\item[(i)]
\begin{equation}
\label{eq inv PH_n_1}
\left(\frac{z(b)}{b-a}+\xi s + o(s)\right)\frac{d\mathcal{T}^s}{dt}\langle x,z \rangle_n(t)
=\frac{z(b)}{b-a}, \ \text{for some constant } \xi;
\end{equation}
	
\item[(ii)]
\begin{multline}
\label{eq inv PH_n_2}
\frac{d \mathcal{Z}^s}{dt}\langle x,z \rangle_n(t)
= L\left(\mathcal{T}^s\langle x,z \rangle_n(t),\mathcal{X}^s\langle x,z \rangle_n(t),
\frac{d\mathcal{X}^s}{d\mathcal{T}^s}\langle x,z \rangle_n(t),\ldots\right.\\
\left.\ldots,\frac{d^n\mathcal{X}^s}{d(\mathcal{T}^s)^n}\langle x,z \rangle_n(t),
\mathcal{Z}^s\langle x,z \rangle_n(t)\right)\frac{d\mathcal{T}^s}{dt}\langle x,z \rangle_n(t),
\end{multline}
\end{enumerate}
where
\begin{equation}\label{deriv X^s}
\frac{d\mathcal{X}^s}{d\mathcal{T}^s}\langle x,z \rangle_n(t)
=\frac{\frac{d\mathcal{X}^s}{dt}\langle x,z \rangle_n(t)}{
\frac{d\mathcal{T}^s}{dt}\langle x,z \rangle_n(t)} \text{ and }
\frac{d^i\mathcal{X}^s}{d(\mathcal{T}^s)^i}\langle x,z \rangle_n(t)
=\frac{\frac{d}{dt}\left(\frac{d^{i-1}\mathcal{X}^s}{d(\mathcal{T}^s)^{i-1}}
\langle x,z \rangle_n(t)\right)}{\frac{d\mathcal{T}^s}{dt}\langle x,z \rangle_n(t)}
\end{equation}
for $i=2,\ldots,n$.
\end{definition}

Next we present the main result of this paper.

\begin{theorem}[Noether's Theorem for problem \eqref{PH_n}]
\label{thm noether PH_n}
If problem \eqref{PH_n} is invariant in the sense
of Definition~\ref{def inv_PH_n}, then the quantity
\begin{multline*}
\sum_{i=1}^n\psi_i(t) X_{i-1}\langle x,z \rangle_n(t)
+\psi_z(t) Z\langle x,z \rangle_n(t)\\
-\left(\sum_{i=1}^n\psi_i(t) x^{(i)}(t)
+\psi_z(t)L\langle x,z \rangle_n(t)\right)T\langle x,z \rangle_n(t)
\end{multline*}
is constant in $t$ along every extremal to problem \eqref{PH_n}, where
\begin{equation*}
\begin{gathered}
T=\frac{\partial \mathcal{T}^s}{\partial s}\biggm\vert_{s=0},
\quad X_0=\frac{\partial \mathcal{X}^s}{\partial s}\biggm\vert_{s=0},
\quad Z=\frac{\partial \mathcal{Z}^s}{\partial s}\biggm\vert_{s=0},\\
X_i=\frac{d }{dt}X_{i-1}-x^{(i)}\frac{d}{dt}\left(\frac{\partial\mathcal{T}^s}{\partial s}
\bigg\vert_{s=0}\right)
\quad \text{for } i=1,\ldots, n-1,
\end{gathered}
\end{equation*}
$\psi_i$ is defined by \eqref{multipliers}
and $\psi_z(t)=e^{\int_t^b\frac{\partial L}{\partial z}d\theta}$.
\end{theorem}

\begin{proof}
As before, we deal with problem \eqref{PH_n} in its equivalent
optimal control form \eqref{PH_n_OC}. We now prove that if problem
\eqref{PH_n} is invariant in the sense of Definition~\ref{def inv_PH_n},
then \eqref{PH_n_OC} is invariant in the sense of Definition~\ref{def inv OC}.
First, observe that if \eqref{eq inv PH_n_1} holds, then \eqref{eq inv OC 1}
holds for \eqref{PH_n_OC} with $f\equiv 0$ and $\phi(x_0, \ldots, x_{n-1},z)=z$.
Second, note that the control system of \eqref{PH_n_OC} defines
$\mathcal{U}^s:=\frac{d\mathcal{X}_{n-1}^s}{d\mathcal{T}^s}$
and $\mathcal{X}_i^s:=\frac{d\mathcal{X}_{i-1}^s}{d\mathcal{T}^s}$, that is,
\begin{equation}
\label{eq inv}
\begin{cases}
\frac{d\mathcal{X}_{i-1}^s}{dt}\langle x,z \rangle_n(t)
=\mathcal{X}_i^s\langle x,z \rangle_n(t)\frac{d\mathcal{T}^s}{dt}\langle x,z
\rangle_n(t), \quad i=1, \ldots, n-1,\\
\frac{d \mathcal{X}_{n-1}^s}{dt}\langle x,z \rangle_n(t)
=\mathcal{U}^s\langle x,z \rangle_n(t)\frac{d\mathcal{T}^s}{dt}\langle x,z \rangle_n(t).
\end{cases}
\end{equation}
This means that if \eqref{eq inv PH_n_2} and \eqref{eq inv} hold, then there
is also invariance in the sense of \eqref{eq inv OC 2} and problem \eqref{PH_n_OC}
is invariant in the sense of Definition~\ref{def inv OC}. This invariance gives
conditions to apply Theorem~\ref{thm opt cont Noether} to problem \eqref{PH_n_OC},
which assures that the quantity
\begin{multline*}
(b-t)\xi+\sum_{i=1}^{n}\psi_i(t) X_{i-1}\langle x,z \rangle_n(t)
+ \psi_z(t) Z\langle x,z \rangle_n(t)\\
- \left[\sum_{i=1}^{n}\psi_i(t) x_i(t)+\psi_z(t) L\langle x,z \rangle_n(t)
+ \frac{\phi(x(b))}{b-a}\right]T\langle x,z \rangle_n(t),
\end{multline*}
where $X_i=\frac{\partial}{\partial s} \frac{d^{i}\mathcal{X}^s}{d(\mathcal{T}^s)^{i}}\Big\vert_{s=0}$
is constant in $t$ along every Pontryagin extremal of problem \eqref{PH_n_OC}. This means that the quantity
\begin{multline*}
(b-t)\xi-\frac{\phi(x(b))}{b-a}T\langle x,z \rangle_n(t)
+\sum_{i=1}^{n}\psi_i(t) X_{i-1}\langle x,z \rangle_n(t)
+ \psi_z(t) Z\langle x,z \rangle_n(t)\\
- \left[\sum_{i=1}^{n}\psi_i(t) x^{(i)}(t)+\psi_z(t) L\langle x,z \rangle_n(t) \right]
T\langle x,z \rangle_n(t)
\end{multline*}
is constant in $t$ along every extremal of problem \eqref{PH_n}.
Observe that $X_0=\frac{\partial \mathcal{X}^s}{\partial s}\big\vert_{s=0}$,
which together with \eqref{deriv X^s} leads to
\begin{align*}
X_i&=\frac{\partial}{\partial s} \frac{d^{i}\mathcal{X}^s}{d(\mathcal{T}^s)^{i}}\bigg\vert_{s=0}
=\frac{\partial}{\partial s} \left( \frac{\frac{d}{dt}\left(
\frac{d^{i-1}\mathcal{X}^s}{d(\mathcal{T}^s)^{i-1}}\right)}{
\frac{d\mathcal{T}^s}{dt}}\right)\Bigg\vert_{s=0}\\
&= \frac{d}{dt}\left(\frac{\partial}{\partial s}
\frac{d^{i-1}\mathcal{X}^s}{d(\mathcal{T}^s)^{i-1}}\bigg\vert_{s=0}\right)-x^{(i)}
\frac{d}{dt}\left(\frac{\partial\mathcal{T}^s}{\partial s} \bigg\vert_{s=0}\right)\\
&=\frac{d }{dt}X_{i-1}-x^{(i)}\frac{d}{dt}\left(
\frac{\partial\mathcal{T}^s}{\partial s} \bigg\vert_{s=0}\right).
\end{align*}
To end the proof we only need to prove that the quantity
\begin{equation}
\label{eq:constant}
(b-t)\xi - \frac{z(b)}{b-a}T\langle x,z \rangle_n(t)
\end{equation}
is a constant. From the invariance condition \eqref{eq inv PH_n_1}, we know that
\begin{equation*}
\left(z(b)+\xi (b-a) s + o(s)\right)
\frac{d\mathcal{T}^s}{dt}\langle x,z \rangle_n(t) =z(b).
\end{equation*}
Integrating from $a$ to $t$ we conclude that
\begin{multline}
\label{exp:prv:ad}
\left(z(b)+\xi (b-a) s + o(s)\right)\mathcal{T}^s\langle x,z \rangle_n(t)\\
=z(b)(t-a)+\left(z(b)+\xi (b-a) s + o(s)\right)\mathcal{T}^s\langle x,z \rangle_n(a).
\end{multline}
Differentiating \eqref{exp:prv:ad} with respect to $s$, and then putting $s=0$, we obtain:
\begin{equation}
\label{eq:relation}
\xi (b-a) t + z(b) T\langle x,z \rangle_n(t) =\xi(b-a)a+z(b)T\langle x,z\rangle_n(a).
\end{equation}
We conclude from \eqref{eq:relation} that expression \eqref{eq:constant} is the constant
$$
(b-a)\xi - \frac{z(b)T\langle x,z\rangle_n(a)}{b-a}.
$$
The proof is complete.
\end{proof}


\section{Conclusion}
\label{sec:conc}

We investigated the higher-order variational problem of Herglotz
from an optimal control point of view. The higher-order
generalized Euler--Lagrange equation and the transversality conditions
proved in \cite{Simao+NM+Torres2013} were obtained in the wider class
of piecewise admissible functions. Moreover, we proved two important new results:
a DuBois--Reymond necessary condition and Noether's theorem
for higher-order variational problems of Herglotz type.


\section*{Acknowledgements}

This work was supported by Portuguese funds through the
\emph{Center for Research and Development in Mathematics and Applications} (CIDMA)
and the \emph{Portuguese Foundation for Science and Technology}
(``FCT --- Funda\c{c}\~ao para a Ci\^encia e a Tecnologia''),
within project UID/MAT/04106/2013. The authors are grateful
to an anonymous referee for several comments and suggestions.



\medskip

Received September 2014; revised July 2015.

\medskip



\begin{thebibliography}{99}
	
\bibitem{MR3072684} (MR3072684) [10.1016/S0034-4877(13)60034-8]
\newblock G. S. F. Frederico\ and\ D. F. M. Torres,
\newblock \emph{Fractional isoperimetric Noether's theorem in the Riemann-Liouville sense},
\newblock Rep. Math. Phys. \textbf{71}, no.~3, 291--304 (2013)
\newblock {\tt arXiv:1205.4853}

\bibitem{Herglotz1930}
\newblock G. Herglotz,
\newblock Ber\"uhrungstransformationen,
\newblock \emph{Lectures at the University of G\"ottingen}, G\"ottingen (1930)

\bibitem{MR2316829} (MR2316829)
\newblock S. Lenhart\ and\ J. T. Workman,
\newblock {\it Optimal control applied to biological models},
\newblock Chapman \& Hall/CRC, Boca Raton, FL, 2007.

\bibitem{Noether1918}
\newblock E. Noether,
\newblock Invariante Variationsprobleme,
\newblock \emph{Nachr. v. d. Ges. d. Wiss. zu Göttingen}, 235--257 (1918)

\bibitem{Pontryagin} (MR166037)
\newblock L. S. Pontryagin, V. G. Boltyanskii, R. V. Gamkrelidze\ and\ E. F. Mishchenko,
\newblock \emph{The mathematical theory of optimal processes.}
\newblock Interscience Publishers, John Wiley and Sons Inc, New York, London (1962)

\bibitem{Simao+NM+Torres2013} (MR3286693) [10.1007/s10013-013-0048-9]
\newblock S. P. S. Santos, N. Martins\ and\ D. F. M. Torres,
\newblock \emph{Higher-order variational problems of Herglotz type},
\newblock Vietnam J. Math. {\bf 42} (2014), no.~4, 409--419.
\newblock {\tt arXiv:1309.6518}

\bibitem{MyID:313} [10.1007/978-3-319-20352-2_7]
\newblock S. P. S. Santos, N. Martins\ and\ D. F. M. Torres,
\newblock \emph{An optimal control approach to Herglotz variational problems},
in ``Optimization in the Natural Sciences''
(eds. A.~Plakhov, T.~Tchemisova and A.~Freitas),
\newblock Communications in Computer and Information Science,
Vol. 499, Springer, (2015), 107--117.
\newblock {\tt arXiv:1412.0433}

\bibitem{Torres2001} (MR1901565) [10.1007/3-540-45606-6_20]
\newblock D. F. M. Torres,
\newblock \emph{Conservation laws in optimal control},
\newblock in {\it Dynamics, bifurcations, and control (Kloster Irsee, 2001)}, 287--296,
\newblock Lecture Notes in Control and Inform. Sci., 273, Springer, Berlin, (2002)

\bibitem{Torres2002} [10.3166/ejc.8.56-63]
\newblock D. F. M. Torres,
\newblock \emph{On the Noether theorem for optimal control},
\newblock European Journal of Control {\bf 8}, no.~1 , 56--63 (2002)

\bibitem{Torres2004} (MR2040245)
\newblock D. F. M. Torres,
\newblock Quasi-invariant optimal control problems,
\newblock \emph{Port. Math. (N.S.)} {\bf 61}, no.~1, 97--114 (2004)
\newblock {\tt arXiv:math/0302264}

\bibitem{MR2098297} (MR2098297) [10.3934/cpaa.2004.3.491]
\newblock D. F. M. Torres,
\newblock \emph{Proper extensions of Noether's symmetry theorem
for nonsmooth extremals of the calculus of variations},
\newblock Commun. Pure Appl. Anal. \textbf{3}, no.~3, 491--500 (2004)

\bibitem{vanBrunt} (MR2004181)
\newblock B. van Brunt,
\newblock \textit{The calculus of variations},
\newblock Universitext,
\newblock Springer-Verlag, New York (2004)

\end{thebibliography}
\end{document}